\documentclass[a4paper,12pt]{article}
\usepackage{}

\usepackage{amsfonts}
\usepackage{amscd,color}
\usepackage{amsmath,amsfonts,amssymb,amscd}
\usepackage{indentfirst,graphicx,epsfig}
\usepackage{graphicx}
\input{epsf}
\usepackage{graphicx}
\usepackage{epstopdf}
\usepackage{caption}

\newtheorem{thm}{Theorem}[section]

\newtheorem{conj}[thm]{Conjecture}
\newtheorem{lem}[thm]{Lemma}

\newtheorem{rem}[thm]{Remark}
\newtheorem{cor}[thm]{Corollary}
\newtheorem{pro}[thm]{Proposition}
\newenvironment {proof} {\noindent{\em Proof.}}{\hspace*{\fill}$\Box$\par\vspace{4mm}}
\newcommand{\ml}{l\kern-0.55mm\char39\kern-0.3mm}

\title{\textbf{Strong rainbow disconnection in graphs\footnote{Supported by NSFC No.11871034 and 11531011.}}}
\author{{\small Xuqing Bai, Xueliang Li } \\
{\small  Center for Combinatorics and LPMC}\\
{\small Nankai University, Tianjin 300071, China}\\
{\small Email: baixuqing0@163.com, lxl@nankai.edu.cn}\\
}
\date{}
\begin{document}
\maketitle
\begin{abstract}
Let $G$ be a nontrivial edge-colored connected graph.
An edge-cut $R$ of $G$ is called a {\it rainbow edge-cut}
if no two edges of $R$ are colored with the same color.
For two distinct vertices $u$ and $v$ of $G$, if an edge-cut
separates them, then the edge-cut is called a {\it $u$-$v$-edge-cut}.
An edge-colored graph $G$ is called \emph{strong rainbow
disconnected} if for every two distinct vertices
$u$ and $v$ of $G$, there exists a both rainbow and minimum $u$-$v$-edge-cut 
({\it rainbow minimum $u$-$v$-edge-cut} for short)
in $G$, separating them, and this edge-coloring is called a {\it strong
rainbow disconnection coloring} (srd-{\it coloring} for short) of $G$.
For a connected graph $G$, the \emph{strong rainbow disconnection number}
(srd-{\it number} for short) of $G$, denoted by $\textnormal{srd}(G)$, is 
the smallest number of colors that are needed in order to make $G$ strong rainbow disconnected.

In this paper, we first characterize the graphs with $m$ edges such that $\textnormal{srd}(G)=k$ 
for each $k \in \{1,2,m\}$, respectively, and we also show that the
srd-number of a nontrivial connected graph $G$ equals the maximum srd-number among the blocks of $G$.
Secondly, we study the srd-numbers for the complete $k$-partite graphs, $k$-edge-connected $k$-regular graphs and grid graphs.
Finally, we show that for a connected graph $G$, to compute $\textnormal{srd}(G)$ is NP-hard. 
In particular, we show that it is already NP-complete to decide if $\textnormal{srd}(G)=3$ 
for a connected cubic graph. Moreover, we show that for a given edge-colored (with
an unbounded number of colors) connected graph $G$ it is NP-complete to decide whether $G$ is 
strong rainbow disconnected.


\end{abstract}

\section{Introduction}

All graphs considered in this paper are simple, finite and
undirected. Let $G=(V(G), E(G))$ be a nontrivial connected graph with 
vertex-set $V(G)$ and edge-set $E(G)$. For $v\in V(G)$, let
$d_G(v)$ and $N_G(v)$ denote the $degree$ and the $neighborhood$
of $v$ in $G$ (or simply $d(v)$ and $N(v)$, respectively, when
the graph $G$ is clear from the context). We use $\Delta(G)$ to denote 
the maximum degree of $G$. For any notation or terminology not defined here, 
we follow those used in \cite{BM}.

Let $G$ be a graph with an \emph{edge-coloring} $c$:
$E(G)\rightarrow [k]= \{1, 2,\cdots, k\}$, $k \in \mathbb{N}$, where
adjacent edges may be colored the same. When adjacent edges of $G$
receive different colors under $c$, the edge-coloring $c$ is called
\emph{proper}. The \emph{chromatic index} of $G$, denoted by $\chi'(G)$, 
is the minimum number of colors needed in a proper edge-coloring of $G$.
By a famous theorem of Vizing \cite{V}, one has that
$$\Delta(G) \leq \chi'(G) \leq \Delta(G)+1$$
for every nonempty graph $G$. If $\chi'(G) = \Delta(G)$, then
$G$ is said to be in \emph{Class} $1$; if $\chi'(G) = \Delta(G) + 1$, 
then $G$ is said to be in \emph{Class} $2$.

As we know that there are two ways to study the connectivity of a graph,
one way is by using paths and the other is by using cuts.
The rainbow connection using paths has been studied extensively; see
for examples, papers \cite{CJMZ, LSS, LS2} and book \cite{LS1} and the 
references therein. So, it is natural to consider the rainbow edge-cuts 
for the colored connectivity in edged-colored graphs. In \cite{CDHHZ}, 
Chartrand et al. first studied the rainbow edge-cut by introducing the 
concept of rainbow disconnection of graphs. In \cite{BL} we call all of 
them global colorings of graphs since they relate global structural  
property: connectivity of graphs.

An \emph{edge-cut} of a connected graph $G$ is a set $F$ of edges
such that $G-F$ is disconnected. The minimum number of edges in an 
edge-cut of $G$ is the \emph{edge-connectivity} of $G$, denoted by 
$\lambda(G)$. For two distinct vertices $u$ and $v$ of $G$,
let $\lambda_G(u,v)$ (or simply $\lambda(u,v)$ when the graph $G$ 
is clear from the context) denote the minimum number of edges in an 
edge-cut $F$ such that $u$ and $v$ lie in different components of $G-F$, and
this kind of edge-cut $F$ is called a {\it minimum $u$-$v$-edge-cut}.
A \emph{$u$-$v$-path} is a path with ends $u$ and $v$. The
following proposition presents an alternate interpretation of
$\lambda(u,v)$ (see \cite{EFS}, \cite{FF}).
\begin{pro}\label{M}
\emph{For every two distinct vertices $u$ and $v$ in a graph $G$,
\emph{$\lambda(u,v)$} is equal to the maximum number of pairwise
edge-disjoint $u$-$v$-paths in $G$}.
\end{pro}

An edge-cut $R$ of an edge-colored connected graph $G$ is called a 
\emph{rainbow edge-cut} if no two edges in $R$ are colored with the 
same color. Let $u$ and $v$ be two distinct vertices of $G$.
A \emph{rainbow $u$-$v$-edge-cut} is a rainbow edge-cut $R$ of $G$ 
such that $u$ and $v$ belong to different components of $G-R$.
An edge-colored graph $G$ is called \emph{rainbow~disconnected} if
for every two distinct vertices $u$ and $v$ of $G$, there exists a
rainbow $u$-$v$-edge-cut in $G$, separating them.
In this case, the edge-coloring is called a
\emph{rainbow disconnection coloring} (rd-{\it coloring} for short) of $G$.
The \emph{rainbow disconnection number} (or rd-\emph{number} for short) of
$G$, denoted by $\textnormal{rd}(G)$, is the smallest number of colors that
are needed in order to make $G$ rainbow disconnected.
A $\textnormal{rd}$-coloring with $\textnormal{rd}(G)$ colors is called an
$optimal$ rd-\emph{coloring} of $G$.

Remember that in the above Menger's famous result of Proposition \ref{M},
only minimum edge-cuts play a role,
however, in the definition of rd-colorings we only requested the existence
of a $u$-$v$-edge-cut between a pair of vertices $u$ and $v$, which could be
any edge-cut (large or small are both OK).
This may cause the failure of a colored version of such a nice Min-Max result of
Proposition \ref{M}. In order to overcome this problem, we will introduce
the concept of strong rainbow disconnection in graphs, with a hope to set
up the colored version of the so-called Max-Flow Min-Cut Theorem.

An edge-colored graph $G$ is called \emph{strong rainbow disconnected} if for
every two distinct vertices $u$ and $v$ of $G$, there exists a both rainbow
and minimum $u$-$v$-edge-cut ({\it rainbow minimum $u$-$v$-edge-cut} for short)
in $G$, separating them.
In this case, the edge-coloring is called a \emph{strong rainbow disconnection
coloring} (srd-{\it coloring} for short) of $G$.
For a connected graph $G$, we similarly define the
\emph{strong rainbow disconnection number}(srd-\emph{number} for short)
of $G$, denoted by $\textnormal{srd}(G)$, as the smallest number of colors
that are needed in order to make $G$ strong rainbow disconnected.
An $\textnormal{srd}(G)$-coloring with $\textnormal{srd}(G)$ colors is called
an $optimal$ srd-\emph{coloring} of $G$.

The remainder of this paper will be organized as follows.
In Section 2, we first obtain some basic results for the \textnormal{srd}-numbers of graphs.
In Section 3, we study the \textnormal{srd}-numbers for some well-known classes of special graphs.
In Section 4, we show that for a connected graph $G$, to compute $\textnormal{srd}(G)$ is NP-hard.
In particular, we show that it is already NP-complete to decide if $\textnormal{srd}(G)=3$ for a connected cubic graph.
Moreover, we show that for a given edge-colored (with an unbounded number of colors) connected graph $G$ it is NP-complete to decide whether $G$ is strong rainbow disconnected.

\section{Some basic results}

Let $G$ be a connected graph. Recall that for a pair of distinct
vertices $x$ and $y$ of $G$, we say that an edge-cut $\partial(X)$ \emph{separates} $x$ and $y$ if $x\in X$ and $y\in V \setminus X$.
We denote by $C_{G}(x,y)$ the minimum cardinality of
such an edge-cut in $G$. Let $X$ be a vertex subset of $G$, and let
$\overline{X}=V(G)\setminus X$.
Then the graph $G/X $ is obtained from $G$ by \emph{shrinking}
$X$ to a single vertex.
A \emph{trivial edge-cut} is one associated with a single vertex.
A \emph{block} of a graph is a maximal connected subgraph of $G$ containing no cut-vertices. The \emph{block decomposition} of $G$ is the set of blocks of $G$.
From definitions, the following inequalities are obvious.
\begin{pro}\label{srd-pro0}
If $G$ is a nontrivial connected graph with edge-connectivity
$\lambda (G)$, upper edge-connectivity $\lambda^+ (G)$ and number $e(G)$ of edges, then
$$\lambda(G)\leq \lambda^+(G)\leq \textnormal{rd}(G) \leq \textnormal{srd}(G)\leq e(G).$$
\end{pro}

Our first question is that the new parameter \textnormal{srd}-number is
really something new, different from \textnormal{rd}-number ?
However, we have not found any connected graph $G$ with
$\textnormal{srd}(G)\neq  \textnormal{rd}(G)$.
So, we pose the following conjecture.

\begin{conj} \label{P}
For any connected graph $G$, $\textnormal{srd}(G)=\textnormal{rd}(G)$.
\end{conj}

In the rest of the paper we will show that for many classes of graphs
the conjecture is true.

In this section, we characterize all those nontrivial connected graphs
with $m$ edges such that $\textnormal{srd}(G)=k$ for each $k \in \{1,2,m\}$, respectively. We first characterize the graphs with $\textnormal{srd}(G)=m$.
The following are two lemmas which we will be used.

\begin{lem}{\upshape \cite{GH}}\label{srd-lem21}
Let $\partial(X)$ be a minimum edge-cut in a graph $G$ separating two vertices $x$ and $y$, where $x\in X$,
and let $\partial(Y)$ be a minimum edge-cut in $G$ separating two vertices $u$ and $v$ of $X$ ($\overline X$), where $y\in Y$.
Then every minimum $u$-$v$-edge-cut in $G/\overline X$ ($G/X$) is a minimum $u$-$v$-edge-cut in $G$.
\end{lem}

It follows from Lemma \ref{srd-lem21} that we have the following result.

\begin{lem}\label{srd-lem-ub}
Let $G$ be a connected graph of order at least $3$. Then $\textnormal{srd}(G)\leq e(G)-1$.
\end{lem}

\begin{proof}
We distinguish the following two cases.

\textbf{Case 1}. There exists at least one pair of vertices having nontrivial minimum edge-cut.

Let $C_G(x,y)$ be a nontrivial minimum $u$-$v$-edge-cut of $G$, where $x,y\in V(G)$, and let $\partial(X)=\min\{C_G(x,y)|x,y\in V(G)\}$.
Suppose that $\partial(X)$ is a nontrivial minimum $x_0$-$y_0$-edge-cut in graph $G$, where $x_0\in X$, and let $\partial(Y)$ be a minimum $u$-$v$-edge-cut in $G$, where $u,v\in X$ and $y_0\in Y$.
By Lemma \ref{srd-lem21}, we get that every minimum $u$-$v$-edge-cut in $G/\overline{X}$ is a minimum $u$-$v$-edge-cut in $G$.
Now we give an edge-coloring $c$ for $G$ by assigning different colors for each edge of $G[X]$ using colors from $[e(G[X])]$ and assigning different colors for each edge of $G[\overline X]$ using colors from $[e(G[\overline X])]$, respectively, and assigning $|\partial(X)|$ new colors for $\partial(X)$.
Note that the set $E_w$ of edges incident with $w$ is rainbow for each vertex $w$ of $G$, and
$|c|=\max\{e(G[X]),e(G[\overline X])\}+|\partial(X)|\leq e(G)-1$ since $e(G[X]),e(G[\overline X])\geq 1$.

We can verify that the coloring $c$ is an srd-coloring of $G$.
Let $p$ and $q$ be two vertices of $G$. If $p$ and $q$ have a nontrivial
minimum edge-cut $C_G(p,q)$ in $G$, then $|C_G(p,q)|\geq |\partial(X)|$.
Suppose that $p\in X$ and $q\in \overline{X}$. Without loss of generality, let $d(p)\leq d(q)$.
If $d(p)<|\partial(X)|$, then the set $E_p$ of edges incident with $p$ is a rainbow minimum $p$-$q$-edge-cut in $G$ under the coloring $c$; if $|\partial(X)|\leq d(p)\leq d(q)$, then the $\partial(X)$ is a rainbow minimum $p$-$q$-edge-cut in $G$ under the coloring $c$.
If $p,q\in X$ ($\overline{X}$), then the minimum $p$-$q$-edge-cut in $G/\overline{X}$ ($G/X$)
is a rainbow minimum $p$-$q$-edge-cut in $G$ since the colors of the edges in
graph $G/\overline{X}$ ($G/X$) are different from each other under the
restriction of coloring $c$.

\textbf{Case 2}. For any two vertices of $G$, there are only trivial minimum edge-cut.

If $G$ is a tree, then $\textnormal{srd}(G)=1$. Obviously, $\textnormal{srd}(G)\leq e(G)-1$ since $G$ is a connected graph with $n\geq 3$. Otherwise,
we give a proper edge-coloring for $G$ using $n-1$ colors. Since $G$ is not a tree, we have $n-1\leq e(G)-1$. For any two vertices $p$, $q$ of $G$, without loss of generality, let $d(p)\leq d(q)$, the set $E_p$ of edges incident with $p$ is a rainbow minimum $p$-$q$-edge-cut in $G$.
\end{proof}

By Lemma \ref{srd-lem-ub}, we immediately obtain the following result.
\begin{cor}
Let $G$ be a connected graph. Then $\textnormal{srd}(G)=e(G)$ if and only if $G=P_2$.
\end{cor}

Next, we further characterize the graphs $G$ with $\textnormal{srd}(G)=1$ and $2$, respectively.
We first restate two results as lemmas which characterize the graphs with $\textnormal{rd}(G)=1$ and $2$, respectively.
\begin{lem}{\upshape \cite{CDHHZ}} \label{srd-lem12}
Let $G$ be a nontrivial connected graph. Then $\textnormal{rd}(G)=1$ if and only if $G$ is a tree.
\end{lem}

\begin{lem}{\upshape \cite{CDHHZ}}\label{srd-lem11}
Let $G$ be a nontrivial connected graph. Then $\textnormal{rd}(G)=2$ if and only if each block of $G$ is either $K_2$ or a cycle and at least one block of $G$ is a cycle.
\end{lem}

Furthermore, we obtain the following two results.
\begin{thm}\label{srd-thm-equ1}
Let $G$ be a nontrivial connected graph.
Then $\textnormal{srd}(G)=1$ if and only if $\textnormal{rd}(G)=1$.
\end{thm}

\begin{proof}
First, if $\textnormal{srd}(G)=1$, then we have $1\leq \textnormal{rd}(G)\leq \textnormal{srd}(G)$ by Proposition \ref{srd-pro0}.
Next, if $\textnormal{rd}(G)=1$, then the graph $G$ has no cycle, namely, the $G$ is a tree. We give one color for all edges of $G$.
Obviously, the coloring is an optimal srd-coloring of $G$, and so $\textnormal{srd}(G)=1$ by Proposition \ref{srd-pro0}.
\end{proof}

\begin{thm}\label{srd-thm-equ2}
Let $G$ be a nontrivial connected graph.
Then $\textnormal{srd}(G)=2$ if and only if $\textnormal{rd}(G)=2$.
\end{thm}

\begin{proof}
First, if $\textnormal{srd}(G)=2$, then $G$ has no cycle with a chord by Proposition \ref{srd-pro0}.
Furthermore, if $G$ is a tree, we showed $\textnormal{srd}(G)=1$.
Therefore, each block of $G$ is either a $K_2$ or a cycle and at least one block of $G$ is a cycle.
By Lemma \ref{srd-lem11}, we get $\textnormal{rd}(G)=2$.

Conversely, suppose $\textnormal{rd}(G)=2$. Then each block of $G$ is either a $K_2$ or a cycle and at least one block of $G$ is a cycle.
We can give a 2-edge-coloring $c$ for $G$ as follows.
Choose one edge from each cycle to give color $1$. The remaining edges are assigned color $2$.
One can easily verify that the coloring $c$ is strong rainbow disconnected. Combined with Proposition \ref{srd-pro0}, we have $\textnormal{srd}(G)=2$.
\end{proof}

By Lemmas \ref{srd-lem12} and \ref{srd-lem11}, and Theorems \ref{srd-thm-equ1} and \ref{srd-thm-equ2}, we immediately get the following corollary.
\begin{cor}\label{srd-cor1}
Let $G$ be a nontrivial connected graph.
Then

\textnormal{(i)} $\textnormal{srd}(G)=1$ if and only if $G$ is a tree.

\textnormal{(ii)} $\textnormal{srd}(G)=2$ if and only if each block of $G$ is either a $K_2$ or a cycle and at least one block of $G$ is a cycle.
\end{cor}

Furthermore, we get $\textnormal{srd}(G)=\textnormal{srd}(B)$, where $\textnormal{srd}(B)$ is maximum among all blocks of $G$.
It implies that the study of srd-numbers can be restricted to 2-connected graphs.

\begin{lem}\label{srd-lem-sub}
If $H$ is a block of a graph $G$, then $\textnormal{srd}(H)\leq \textnormal{srd}(G)$.
\end{lem}

\begin{proof}
Let $c$ be an optimal srd-coloring of $G$, and let $u,v$ be two vertices of $H$.
Suppose $R$ is a rainbow minimum $u$-$v$-edge-cut in $G$.
Then $R\cap E(H)$ is a rainbow minimum $u$-$v$-edge-cut in $H$.
Assume that there exists a smaller $u$-$v$-edge-cut $R'$ in $H$. Then there is no $u$-$v$-path in $G\setminus R'$, which is a contradiction with definition of $R$ since $|R'|<|R|$.
Hence, the coloring $c$ restricted to $H$ is an srd-coloring of $H$.
Thus, $\textnormal{srd}(H)\leq \textnormal{srd}(G)$.
\end{proof}

\begin{thm}\label{srd-thm-block}
Let $G$ be a nontrivial connected graph, and $B$ a block of $G$ such that $\textnormal{srd}(B)$ is maximum
among all blocks of $G$. Then $\textnormal{srd}(G)=\textnormal{srd}(B)$.
\end{thm}

\begin{proof}
Let $\{B_1,B_2,\ldots, B_t\}$ be the block decomposition of $G$, and
let $k=\max\{\textnormal{srd}(B_i):i\in[t]\}$.
If $G$ has no cut-vertex, then $G = B_1$ and the result follows.
Hence, we may assume that $G$ has at least one cut-vertex.
By Lemma \ref{srd-lem-sub}, we have $k\leq \textnormal{srd}(G)$.

Let $c_i$ be an optimal srd-coloring of $B_i$.
We define the edge-coloring $c$: $E(G)\rightarrow[k]$ of
$G$ by $c(e) = c_i(e)$ if $e\in E(B_i)$.
Let $u$ and $v$ be two vertices of $G$. If $u,v\in B_i$ $(i\in[t])$,
let $C_{G}(u,v)=C^r_{B_i}(u,v)$, where $C^r_{B_i}(u,v)$ is the rainbow minimum $u$-$v$-edge-cut in $B_i$.
Obviously, $C_{G}(u,v)$ is rainbow under the coloring $c_i$.
Moreover, it is minimum $u$-$v$-edge-cut in $G$. Otherwise, assume that $R$ is a smaller $u$-$v$-edge-cut in $G$. Then $R\cap E(B_i)$ is also a $u$-$v$-edge-cut in $B_i$, which contradicts to the definition of $C^r_{B_i}(u,v)$ since $|R\cap E(B_i)|<|C_{B_i}(u,v)|$.
Hence, the $C_{G}(u,v)$ is a rainbow minimum $u$-$v$-edge-cut in $G$.
Suppose that $u\in B_i$ and $v\in B_j$, where $i<j$ and $i,j\in[t]$.
Let $B_ix_iB_{i+1}x_{i+1}\ldots x_{j-1}B_{j}$ be a unique $B_i$-$B_j$-path in the block-tree of $G$, and let $x_i$ be the cut-vertex between blocks $B_i$ and $B_{i+1}$.
If $u=x_i$ and $v=x_{j-1}$, let $C_{G}(u,v)=\min\{C^r_{B_{i+1}}(x_i,x_{i+1}),\ldots, C^r_{B_{j-1}}(x_{j-2},x_{j-1})\}$.
If $u=x_i$ and $v\neq x_{j-1}$, let $C_{G}(u,v)=\min\{C^r_{B_{i+1}}(x_i,x_{i+1}),\ldots,
C^r_{B_{j-1}}(x_{j-2},x_{j-1}),C^r_{B_j}(x_{j-1},v)\}$.
If $u\neq x_i$ and $v=x_{j-1}$, let $C_{G}(u,v)=\min\{C^r_{B_i}(u,x_i),C^r_{B_{i+1}}(x_i,x_{i+1}),\ldots, C^r_{B_{j-1}}(x_{j-2},x_{j-1})\}$.
If $u\neq x_i$ and $v\neq x_{j-1}$, let $C_{G}(u,v)=\min\{C^r_{B_i}(u,x_i),C^r_{B_{i+1}}(x_i,x_{i+1}),\ldots,
C^r_{B_j}(x_{j-1},v)\}$.
By the connectivity of $G$, we know that $\lambda_G(u,v)=|C_{G}(u,v)|$, and $C_{G}(u,v)$ is rainbow.
Then $C_{G}(u,v)$ is a rainbow minimum $u$-$v$-edge-cut in $G$.
Hence, $\textnormal{srd}(G)\leq k$, and so $\textnormal{srd}(G)=k$.
\end{proof}

\begin{rem} As one has seen that all the above results for the
\textnormal{srd}-number behave the same as for the \textnormal{rd}-number.
This supports Conjecture \ref{P}.
\end{rem}

\section{The srd-numbers of some classes of graphs}

In this section, we investigate the srd-numbers of complete graphs, complete multipartite graphs, regular graphs and grid graphs. Again,
we will see that the srd-number behaves the same as the rd-number.
At first, we restate several results as lemmas which will be used in the sequel.

\begin{lem}{\upshape\cite{ACCGZ}}\label{rd-classs}
Let $G$ be a connected graph. If every connected component of $G_\Delta$
is a unicyclic graph or a tree, and $G_\Delta$ is not a disjoint union of cycles, then $G$ is in Class $1$.
\end{lem}

\begin{lem}{\upshape \cite{CDHHZ}}\label{srd-lem16}
For each integer $n\geq 4$, $\textnormal{rd}(K_n) = n-1$.
\end{lem}

\begin{lem}{\upshape \cite{BCHL}}\label{rd-lem15}
If $G=K_{n_1,n_2,...,n_k}$ is a complete $k$-partite graph with
order $n$, where $k\geq 2$ and $n_1\leq n_2\leq \cdots \leq n_k$,
then
$$\textnormal{rd}(K_{n_1,n_2,...,n_k})=
\begin{cases}
n-n_2,& \text{if $n_1=1$},\\
n-n_1,& \text{if $n_1\geq 2$}.
\end{cases}$$
\end{lem}

\begin{lem}{\upshape \cite{BCHL}}\label{rd-regu}
If $G$ is a connected $k$-regular graph, then $k \leq
\textnormal{rd}(G) \leq k+1$.
\end{lem}

\begin{lem}{\upshape\cite{ACCGZ}}\label{rd-2}
The \textnormal{rd}-number of the grid graph $G_{m,n}$ is as follows.

\textnormal{(i)} For all $n\geq 2$, $\textnormal{rd}(G_{1,n})=\textnormal{rd}(P_n) = 1$.

\textnormal{(ii)} For all $n\geq 3$, $\textnormal{rd}(G_{2,n})=2$.

\textnormal{(iii)} For all $n\geq 4$, $\textnormal{rd}(G_{3,n})=3$.

\textnormal{(iv)} For all $4\geq m\geq n$, $\textnormal{rd}(G_{m,n})=4$.
\end{lem}

First, we get the \textnormal{srd}-number for complete graphs.
\begin{thm}
For each integer $n\geq 2$, $\textnormal{srd}(K_n) = n-1$.
\end{thm}

\begin{proof}
By Proposition \ref{srd-pro0} and Lemma \ref{srd-lem16},
$n-1\leq \textnormal{rd}(K_n)\leq \textnormal{srd}(K_n)$.
It remains to show that there exists an srd-coloring for $K_n$ using $n-1$ colors.
Suppose first that $n\geq 2$ is even.
Let $u$ and $v$ be two vertices of $K_n$, and let $c$ be a
proper edge-coloring of $K_n$ using $n-1$ colors.
Since $\lambda(K_n)=n-1$,
the set $E_u$ of edges incident with $u$ is a rainbow minimum $u$-$v$-edge-cut in $G$.
Next suppose $n\geq 3$ is odd.
We give the same edge-coloring for graph $G$ as the coloring in Lemma \ref{srd-lem16}.
We now restate it as follows.
Let $x$ be a vertex of $K_n$ and $K_{n-1}=K_n-x$.
Then $K_{n-1}$ has a proper edge-coloring $c$ using $n-2$ colors
since $n-1$ is even.
Now we extend an edge-coloring $c$ of $K_{n-1}$ to $K_n$ by assigning color $n-1$ for each edge incident with vertex $x$.
We show that the $c$ is an srd-coloring of $G$.
Let $u$ and $v$ be two vertices of $K_n$, say $u\neq x$. Then the set $E_u$ of edges incident with $u$ is a rainbow minimum $u$-$v$-edge-cut in $G$ since $\lambda(K_n)=n-1$.
\end{proof}

Then, we give the srd-number for complete multipartite graphs.
\begin{thm}\label{rd-thm1}
If $G=K_{n_1,n_2,...,n_k}$ is a complete $k$-partite graph with
order $n$, where $k\geq 2$ and $n_1\leq n_2\leq \cdots \leq n_k$,
then
$$\textnormal{srd}(K_{n_1,n_2,...,n_k})=
\begin{cases}
n-n_2,& \text{if $n_1=1$},\\
n-n_1,& \text{if $n_1\geq 2$}.
\end{cases}$$
\end{thm}

\begin{proof}
It remains to prove that $\textnormal{srd}(G) \leq n-n_2$ for $n_1=1$, and $\textnormal{srd}(G) \leq n-n_1$ for $n_1\geq 2$ by Proposition \ref{srd-pro0} and Lemma \ref{rd-lem15}.
Let $V_1,V_2,\ldots V_k$
be the $k$-partition of the vertices of $G$, and
$V_i=\{v_{i,1},v_{i,2},\ldots, v_{i,n_i}\}$ for each $i\in [k]$. We distinguish the following two cases.

\textbf{Case 1.} $n_1=1$.

First, we have $V_1=\{v_{1,1}\}$ and $d(v_{1,1})=n-1$. Let
$H=G-\{v_{1,1}\}$. Then $\Delta(H) = n-n_2-1$.
Then, we construct a proper edge-coloring $c_0$ of $H$ using colors from
$[\Delta(H) + 1]$. For each vertex $x \in V(H)$, since $d_{H}(x)\leq
\Delta(H)$, there is an $a_x\in [\Delta(H) + 1]$ such that $a_x$ is
not assigned to any edge incident with $x$ in $H$. Since
$E(G)=E(H)\cup \{v_{1,1}x \mid x\in N_G(v_{1,1})\}$, we now extend the
edge-coloring $c_0$ of $H$ to an edge-coloring $c$ of $G$ by
assigning $c(v_{1,1}x)=a_x$ for every vertex $x\in N_G(v_{1,1})$. Note that the set $E_x$ of edges incident with $x$ is a rainbow set for each vertex $x\in V(G)\setminus v_{1,1}$ in $G$.
Suppose $p$ and $q$ are two vertices of $G$.
If $p\in V_i$ and $q\in V_j$ $(1\leq i<j\leq t)$, then the set $E_q$ of edges incident with $q$ is a rainbow minimum $p$-$q$-edge-cut in $G$ since $\lambda_G(p,q)=n-n_j$.
If $p,q\in V_i$, then the set $E_q$ of edges incident with $q$ is a rainbow minimum $p$-$q$-edge-cut in $G$ since $\lambda_G(p,q)=n-n_i$.
Hence, we obtain $\textnormal{srd}(G)\leq \Delta(H)+1=n-n_2$.

\textbf{Case 2.} $n_1\geq 2$.

Pick a vertex $u$ of $V_1$ and let $F=G-u$. Then $\Delta(F) = n-n_1$
since $n_1\geq 2$ and $F_\Delta=G[V_1-u]$. It follows from Lemma
\ref{rd-classs} that $F$ is in Class $1$, and so $\chi'(F)=n-n_1$.
Furthermore, for each vertex $x\in N_G(u)$,
we know $d_F(x)\leq \Delta(F)-1= n-n_1-1$. Similar to the argument of Case 1, we can construct an edge-coloring $c$ for $G$ such that the
set $E_x$ of edges incident with $x$ is a rainbow set for each vertex $x\in V(G)\setminus u$ using $n-n_1$ colors.
Suppose $p$ and $q$ are two vertices of $G$.
If $p\in V_i$ and $q\in V_j$ $(1\leq i<j\leq t)$, then the set $E_q$ of edges incident with $q$ is a rainbow minimum $p$-$q$-edge-cut in $G$ since $\lambda_G(p,q)=n-n_j$.
If $p,q\in V_i$ $(i\in[t])$, say $q\neq u$, then the set $E_q$ of edges incident with $q$ is a rainbow minimum $p$-$q$-edge-cut in $G$ since $\lambda_G(p,q)=n-n_i$.
Hence, $\textnormal{srd}(G)\leq n-n_1$.
\end{proof}

For regular graphs, we only study the srd-number of $k$-edge-connected $k$-regular graphs. Moreover, we obtain that $\textnormal{srd}(G)=k$ if and only if $\chi'(G)=k$ for a $k$-edge-connected $k$-regular graph $G$, where $k$ is odd.

\begin{lem}{\upshape \cite{BHL}}\label{rd-lem22}
Let $k$ be an odd integer, and $G$ a $k$-edge-connected $k$-regular graph of order $n$. Then $\chi'(G)=k$ if and only if $\textnormal{rd}(G)=k$.
\end{lem}

\begin{thm}\label{srd-thm-reg0}
Let $G$ be a $k$-edge-connected $k$-regular graph.
Then $k\leq \textnormal{srd}(G)\leq \chi'(G)$.
\end{thm}

\begin{proof}
It follows from Proposition \ref{srd-pro0} that $\textnormal{srd}(G)\geq k$. Let $u$, $v$ be two vertices of $G$. Using the fact that $G$ is a $k$-edge-connected $k$-regular graph, one may easily verify that the set $E_v$ of edges incident with $v$ is a rainbow minimum $u$-$v$-edge-cut under a proper edge-coloring of $G$.
Hence, $\textnormal{srd}(G)\leq \chi'(G)$.
\end{proof}

\begin{thm}\label{srd-thm-reg}
Let $k$ be an odd integer, $G$ a $k$-edge-connected $k$-regular graph. Then $\textnormal{srd}(G)=k$ if and only if $\textnormal{rd}(G)=k$.
\end{thm}

\begin{proof}
First, suppose $\textnormal{srd}(G)=k$. Since $\lambda(G)=k$, we have $\textnormal{rd}(G)=k$ by Proposition \ref{srd-pro0}.
Conversely, if $\textnormal{rd}(G)=k$, then we have $\textnormal{srd}(G)=k$ by Proposition \ref{srd-pro0} and Lemma \ref{rd-lem22} and Theorem \ref{srd-thm-reg0}.
\end{proof}

By Lemma \ref{rd-lem22} and Theorem \ref{srd-thm-reg}, we immediately get the following corollary.
\begin{cor}\label{srd-lem20}
Let $k$ be an odd integer, $G$ a $k$-edge-connected $k$-regular graph.
Then $\textnormal{srd}(G)=k$ if and only if $\chi'(G)=k$.
\end{cor}

The Cartesian product $G\Box H$ of two vertex-disjoint graphs $G$ and $H$ is the graph with vertex-set $V(G)\times V(H)$, where $(u,v)$ is adjacent to $(x,y)$ in $G\Box H$ if and only if either
$u=x$ and $vy\in E(H)$ or $ux\in E(G)$ and $v=y$. We consider the $m\times n$ grid graph $G_{m,n} = P_m \Box P_n$, which consists of $m$ horizontal paths $P_n$ and $n$ vertical paths $P_m$.
Now we determine the srd-number for grid graphs.

\begin{thm}\label{srd-thm0}
The $\textnormal{srd}$-number of the grid graph $G_{m,n}$ is as follows.

\textnormal{(i)} For all $n\geq 2$, $\textnormal{srd}(G_{1,n})=\textnormal{srd}(P_n) = 1$.

\textnormal{(ii)} For all $n\geq 3$, $\textnormal{srd}(G_{2,n})=2$.

\textnormal{(iii)} For all $n\geq 4$, $\textnormal{srd}(G_{3,n})=3$.

\textnormal{(iv)} For all $4\geq m\geq n$, $\textnormal{srd}(G_{m,n})=4$.
\end{thm}

\begin{proof}
First, it follows from Proposition \ref{srd-pro0} and Lemma \ref{rd-2} that the lower bounds on $\textnormal{srd}(G_{m,n})$ in (i)-(iv) hold. It remains to show that the upper bound on $\textnormal{srd}(G_{m,n})$ in each of (i)-(iv) also holds.

\textnormal{(i)} We get $\textnormal{srd}(G_{1,n})=\textnormal{srd}(P_n) = 1$ by Corollary \ref{srd-cor1}.

For the rest of the proof, the vertices of $G_{m,n}$ are regarded as a matrix. Let $x_{i,j}$ be the vertex in row $i$ and column $j$, where $1\leq i\leq m$ and $1\leq j\leq n$.

\textnormal{(ii)} We give the same edge-coloring $c$ for $G_{2,n}$ ($n\geq 3$) using colors from the elements of $Z_3$ of the integer modulo  3 as in Lemma \ref{rd-2} (ii).
Define the edge-coloring $c$ for $G_{2,n}$: $c\rightarrow Z_3 $, and we now restate it as follows.

$\bullet$ $c(x_{i,j}x_{i,j+1})=i+j+1$ for $1\leq i\leq 2$ and $1\leq j\leq n-1$;

$\bullet$ $c(x_{1,j}x_{2,j})=j$ for $1\leq j\leq n-1$.

One can verify that the coloring $c$ is an srd-coloring for $G_{2,n}$.
Let $u$ and $v$ be two vertices of $G_{2,n}$.
If $u$ and $v$ are in different columns,
then two parallel edges between $u$ and $v$ joining vertices in the same two columns form a rainbow minimum $u$-$v$-edge-cut in $G_{2,n}$ since $\lambda(u,v)=2$.
Suppose $u$ and $v$ are in the same column. Because the set $E_u$ of edges incident with $u$ is rainbow and $\lambda(u,v)=d(u)=d(v)$, the set $E_u$ of edges incident with $u$ is a rainbow minimum $u$-$v$-edge-cut in $G_{2,n}$.

\textnormal{(iii)} Give the same edge-coloring $c$ as for the graph $G_{3,n}$ ($n\geq 3$) in Lemma \ref{rd-2} (iii). Again we use the elements of $Z_3$  as the colors here.
Define the edge-coloring $c$ for $G_{3,n}$: $c\rightarrow Z_3 $ as follows.

$\bullet$ $c(x_{i,j}x_{i,j+1})=i+j+1$ for $1\leq i\leq 3$ and $1\leq j\leq n-1$;

$\bullet$ $c(x_{1,j}x_{2,j})=j$ for $1\leq j\leq n-1$;

$\bullet$ $c(x_{2,j}x_{3,j})=j+2$ for $1\leq j\leq n-1$.

Now we show that the coloring $c$ is an srd-coloring of $G_{3,n}$.
Observe that the set $E_x$ of edges incident with $x$ is rainbow for each vertex $x$ with $d(x)\leq 3$ in $G_{3,n}$ under the coloring $c$.
Let $u$ and $v$ be two vertices of $G_{3,n}$.
If $u$ and $v$ have at most one vertex with degree $4$, without loss of generality, $2\leq d(u)\leq d(v)\leq 4$, then the set $E_u$ of edges incident with $u$ is a rainbow minimum $u$-$v$-edge-cut in $G_{3,n}$ since $\lambda(u,v)=d(u)$.
If $d(u)=d(v)=4$, then three parallel edges between $u$ and $v$ joining vertices in the same two columns form a rainbow minimum $u$-$v$-edge-cut in $G_{3,n}$ since $\lambda(u,v)=3$.

\textnormal{(iv)}
For the graph $G_{m,n}$ ($4\leq m\leq n$), because $G_{m,n}$ is bipartite and $\Delta(G_{m,n})=4$, there exists a proper edge-coloring $c$ using 4 colors. Now we show that the $c$ is an
srd-coloring of $G_{m,n}$.
Let $u$ and $v$ be two vertices of $G_{m,n}$.
Suppose $d(u)\leq d(v)$. Then the set $E_u$ of edges incident with $u$ is a rainbow minimum $u$-$v$-edge-cut in $G_{m,n}$ ($4\leq m\leq n$) since $\lambda(u,v)=d(u)$.
\end{proof}

\section{Hardness results}

First, we show that our problem is in NP for any fixed integer $k$.

\begin{lem}\label{srd-lem-P}
For a fixed positive integer $k$, given a $k$-edge-colored graph $G$,
deciding whether $G$ is a strong rainbow disconnected under the coloring is
in $P$.
\end{lem}

\begin{proof}
Let $n$, $m$ be the number of the vertices and edges of $G$, respectively.
Let $u$, $v$ be two vertices of $G$.
Because $G$ has at most $k$ colors, we have at most $\sum_{l=1}^k {m\choose l}$ rainbow edge subsets in $G$, denoted the set of the subsets by $\mathcal{S}$. One can see that this number is upper bounded by a polynomial in $m$ when $k$ is a fixed integer (say $km^k$, roughly speaking). Given a rainbow subset of edges $S\in \mathcal{S}$, it is checkable in polynomial time to decide whether $S$ is a $u$-$v$-edge-cut of $G$, just to see whether $u$ and $v$ lie in different components of $G\setminus S$, and the number of components is a polynomial in $n$.
If each rainbow subset in $\mathcal{S}$ is not a $u$-$v$-edge-cut in $G$,
then the coloring is not an srd-coloring of $G$, which can be checked in polynomial time since the number of such subsets is polynomial many in $m$. Otherwise, let the integer $l_0(\leq k)$ be the minimum size of a $u$-$v$-edge-cut in $G$, and this $l_0$ can be computed in polynomial time. Then, if one of the rainbow subsets of $\mathcal{S}$ is a $u$-$v$-edge-cut of $G$ with size $l_0$, then it is a rainbow minimum $u$-$v$-edge-cut of $G$, which can be done in polynomial time since
the number of such subsets is polynomial many in $m$.
Otherwise, the coloring is not an srd-coloring. Moreover, there are at most ${n\choose 2}$ pairs of vertices in $G$. Since $k$ is an integer, we can deduce that deciding wether a $k$-edge-colored graph $G$ is strong rainbow disconnected can be checked in polynomial time.
\end{proof}

In particular, it is $NP$-complete to determine whether $\textnormal{srd}(G)=3$ for a cubic graph. We first restate the
following result as a lemma.

\begin{lem}{\upshape\cite{BCHL}}\label{srd-lem4}
It is $NP$-complete to determine whether the \textnormal{rd}-number of a cubic is $3$ or $4$.
\end{lem}

\begin{thm}\label{srd-thm1}
It is $NP$-complete to determine whether the \textnormal{srd}-number of a
cubic is $3$ or $4$.
\end{thm}

\begin{proof}
The problem is in NP from Lemma \ref{srd-lem-P}.
Furthermore, we get that it is NP-hard to determine whether the srd-number
of a 3-edge-connected cubic is $3$ or $4$ by Theorem \ref{srd-thm-reg} and
the proof of Lemma \ref{srd-lem4}.
\end{proof}

Lemma \ref{srd-lem-P} tells us that deciding whether a given $k$-edge-colored
graph $G$ is strong rainbow disconnected for a fixed integer $k$ is in P. However, it is NP-complete to decide whether a given edge-colored (with an unbounded number of colors) graph is strong rainbow disconnected.

\begin{thm}\label{srd-thm2}
Given an edge-colored graph $G$ and two vertices
$s, t$ of $G$, deciding whether there is a rainbow minimum $s$-$t$-edge-cut is NP-complete.
\end{thm}

\begin{proof}
Clearly, the problem is in NP, since for a graph $G$ checking whether a given set of edges is a rainbow minimum $s$-$t$-edge-cut in $G$ can be done in polynomial time, just to see whether it is an $s$-$t$-edge-cut and it has the minimum size $\lambda_G(s,t)$ by solving the maximum flow problem. We exhibit a polynomial reduction from the problem 3SAT.
Given a 3CNF for $\phi=\wedge_{i=1}^{m}c_i$ over variables
$x_1,x_2,\ldots,x_n$, we construct a graph $G_\phi$
with two special vertices $s,t$ and an edge-coloring $f$ such that
there is a rainbow minimum $s$-$t$-edge-cut in
$G_{\phi}$ if and only if $\phi$ is satisfiable.

We define $G_\phi$ as follows:

\[
\begin{aligned}
V(G_\phi)=
&    & &\{s,t\}
\cup\{x_{i,0},x_{i,1}|i\in [n]\}
\cup \{c_{i,j}|i\in[m],j\in\{0,1,2,3\}\}\\
&\cup& &\{p_{i,j},q_{i,j}|i\in [n],j\in[\ell_i]\}
\cup\{y_i|i\in [5m+1]\},
\end{aligned}
\]
where $\ell_i$ is the number of times of each variable $x_i$ appearing  among the clauses of $\phi$.
\[
\begin{aligned}
E(G_{\phi})=
&    & &\big\{sp_{i,l},sq_{i,l} \ | \ i\in [n],l\in [\ell_i]\big\}\\
&\cup& &\big\{p_{i,l}x_{i,0}, q_{i,l}x_{i,1} \ | \ i\in [n],l\in [\ell_i]\big\}\\
&\cup& &\big\{x_{j,0}c_{i,0}, c_{i,0}c_{i,k}, c_{i,k}x_{j,1} \ | \ \\
&    & & \text{if variable $x_j$ is positive in the $k$-th literal of  clause $c_i$,}\\
&    & &i\in [m], j\in[n], k\in\{1,2,3\}\big\}.\\
&\cup& &\big\{ x_{j,1}c_{i,0}, c_{i,0}c_{i,k}, c_{i,k}x_{j,0} \ | \ \\
&    & & \text{if variable $x_j$ is negative in the $k$-th literal of  clause $c_i$,}\\
&    & &i\in [m], j\in[n], k\in\{1,2,3\}\big\}.\\
&\cup& &\big\{E(K_{6m+2}) \ | \ V(K_{6m+2})=
\{c_{1,0},\ldots,c_{m,0},y_1,\ldots,y_{5m+1},t\}\big\}.
\end{aligned}
\]

The edge-coloring $f$ is defined as follows (see Figure \ref{srd-satisfy}):
\begin{itemize}
\item The edges $\big\{sp_{i,l}, p_{i,l}x_{i,0},sq_{i,l},q_{i,l}x_{i,1} \ | \ i\in [n],l\in [\ell_i]\big\}$ are
colored with a special color $r_{i,l}^0$.
\item The edge $x_{j,0}c_{i,0}$ or $x_{j,1}c_{i,0}$ is
colored with a special color $r_{i,k}$ when $x_j$
is the $k$-th literal of clause $c_i$,
$i\in [m], j\in[n], k\in\{1,2,3\}$.
\item The edge $c_{i,k}x_{j,0}$ or $c_{i,k}x_{j,1}$ is
colored with a special color $r_{i,4}$,
$i\in [m], j\in[n], k\in\{1,2,3\}$.
\item The edge $c_{i,k}c_{i,0}$ is colored with a
special color $r_{i,5}$, $i\in [m], k\in\{1,2,3\}$.
\item The remaining edges are colored with a special color $r_0$.
\end{itemize}

\begin{figure}[!htb]
\centering
\includegraphics[width=0.8\textwidth]{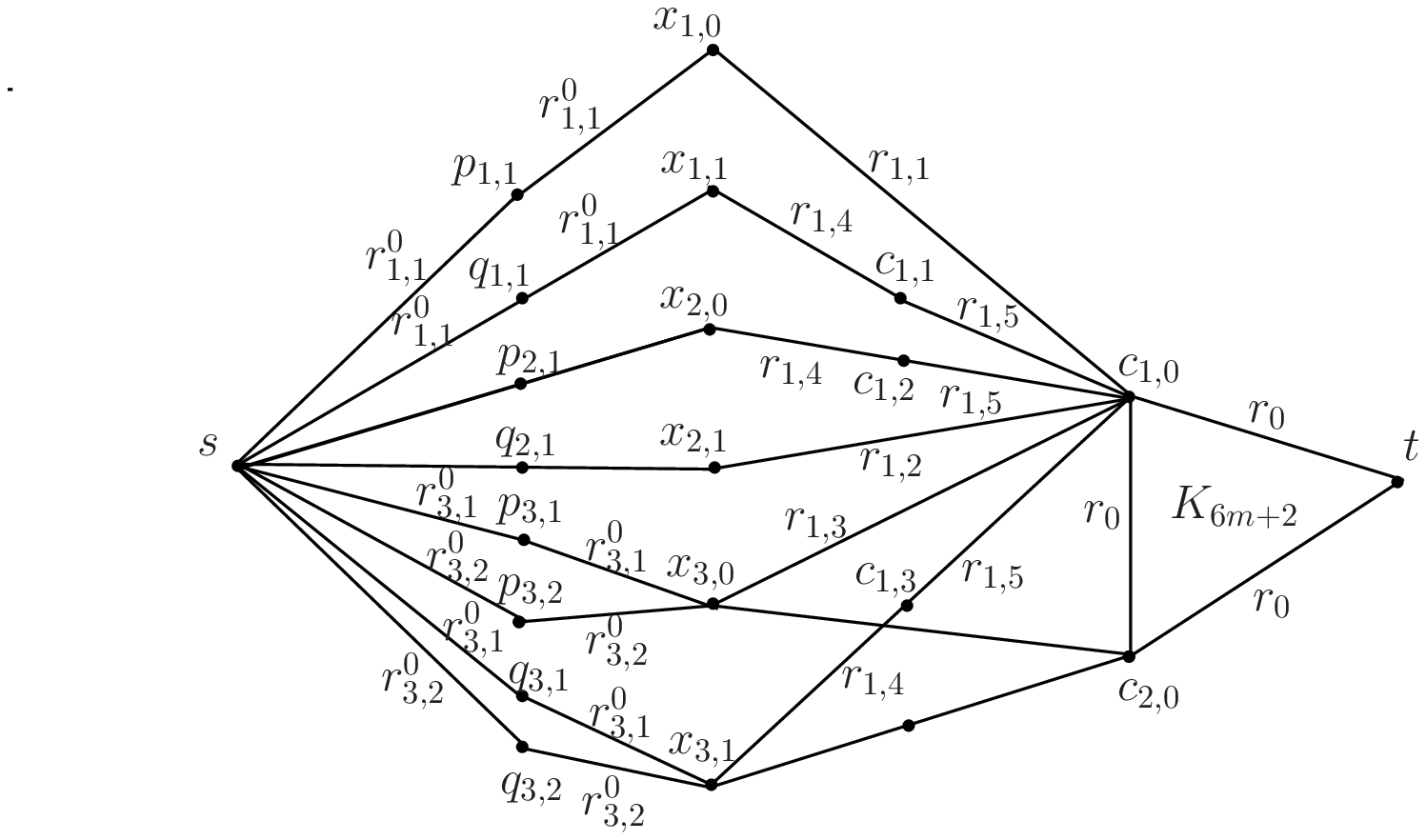}
\caption{The clause $c_1=(x_1,\overline{x_2},x_3)$ and the variable $x_3$ is in clause $c_1$ and $c_2$.}\label{srd-satisfy}
\end{figure}

Now we verify that there is a rainbow minimum $s$-$t$-edge-cut
in $G_{\phi}$ if and only if $\phi$ is satisfiable.

Assume that there exists a rainbow minimum $s$-$t$-edge-cut $S$ in $G_\phi$ under the coloring $f$, and let us show that $\phi$
is satisfiable.
Note that for each $j\in[n]$, $l\in l_j$, if $S$ has an edge in $\{sp_{j,l},p_{j,l}x_{j,0}\}$ (or $\{sq_{j,l},q_{j,l}x_{j,0}\}$), then a rainbow $s$-$x_{j,0}$(or $s$-$x_{j,1}$)-edge-cut in $G[s\cup x_{j,0}\cup \{p_{j,l}|l\in l_j\}]$ is in $S$, and no edge of $\{sq_{j,l},q_{j,l}x_{j,1}|l\in[l_j]\}$ (or $\{sp_{j,l},p_{j,l}x_{j,0}|l\in[l_j]\}$) is in $S$.
Otherwise, it contradicts to the assumption that $S$ is a rainbow minimum $s$-$t$-edge-cut in $G_\phi$.
For each $j\in [n]$, if a rainbow $s$-$x_{j,0}$-edge-cut in $G[s\cup x_{j,0}\cup \{p_{j,l}|l\in l_j\}]$ is in $S$ under the coloring $f$, then set $x_j=0$;
if a rainbow $s$-$x_{j,1}$-edge-cut in $G[s\cup x_{j,1}\cup \{q_{j,l}|l\in l_j\}]$ is in $S$ under the coloring $f$, then set $x_j=1$.
First, we have $|S|=6m$ and $S\subseteq G[V(G_\phi)\setminus\{y_1,\ldots,y_{5m+1},t\}]$.
Moreover, for given $c_{i,0}$ ($i\in[m]$),
we know that $S$ has at most two edges from three paths of length two between $c_{i,0}$ and $\{x_{j,0},x_{j,1}|x_j$ in $c_i$ and $j\in[n]\}$ under the coloring $f$ of $G_{\phi}$.
Suppose, without loss of generality, that the path of length two
between $x_{j,0}$ (or $x_{j,1}$) and $c_{i,0}$ has no
edge belonging to $S$ for some $j\in[n]$. If $x_j$ in $c_i$ is positive,
then there exists a rainbow $s$-$x_{j,1}$-edge-cut with size $\ell_j$ in $G[s\cup x_{j,1}\cup \{q_{j,l}|l\in l_j\}]$ belonging to $S$, where $i\in [m]$, $j\in[n]$.
Then $x_j=1$ and $c_i$ is satisfiable.
If $x_j$ in $c_i$ is negative, then there exists a rainbow $s$-$x_{j,0}$-edge-cut with size $\ell_j$ in $G[s\cup x_{j,1}\cup \{p_{j,l}|l\in l_j\}]$ belonging to $S$, where $i\in [m]$, $j\in[n]$.
Then $x_j=0$ and $c_i$ is satisfiable. Since this is true for each $c_i$ ($i\in [m]$), we get that $\phi$ is satisfiable.

Now suppose $\phi$ is satisfiable, and let us construct a rainbow minimum $s$-$t$-edge-cut in $G_\phi$ under the coloring $f$.
First, there exists a satisfiable assignment of $\phi$.
If $x_j=0$, we put the rainbow $s$-$x_{j,0}$-edge-cut in $G[s\cup x_{j,0}\cup \{p_{j,l}|l\in l_j\}]$ into $S$ for each $j\in[n]$.
If the vertex $x_{j,0}$ is adjacent to $c_{i,0}$,
then let one edge of $c_{i,k}x_{j,1},c_{i,k}c_{i,0}$ be
in $S$ for each $i\in [m], j\in[n], k\in\{1,2,3\}$.
If the vertex $x_{j,0}$ is adjacent to $c_{i,k}$,
then let the edge $x_{j,1}c_{i,0}$ be in $S$ for each
$i\in [m], j\in[n], k\in\{1,2,3\}$.
If $x_j=1$, we put the rainbow $s$-$x_{j,1}$-edge-cut in $G[s\cup x_{j,1}\cup \{q_{j,l}|l\in l_j\}]$ into $S$ for each $j\in[n]$.
If the vertex $x_{j,1}$ is adjacent to $c_{i,0}$,
then let one edge of $c_{i,k}x_{j,0},c_{i,k}c_{i,0}$
be in $S$ for each $i\in [m], j\in[n],k\in\{1,2,3\}$.
If the vertex $x_{j,1}$ is adjacent to $c_{i,k}$,
then let the edge $x_{j,0}c_{i,0}$ be in $S$ for each
$i\in [m], j\in[n], k\in\{1,2,3\}$.
Now we verify that $S$ is indeed a rainbow minimum $s$-$t$-edge-cut.
First, we can verify that $|S|=6m$ and it is a minimum $s$-$t$-edge-cut.
In fact, if a literal of $c_i$ is false,
then one edge colored with $r_i^4$ or $r_i^5$ is in $S$.
Since the three literals of $c_i$ cannot be false
at the same time, we can find a rainbow minimum $s$-$t$-edge-cut in $G_\phi$ under the coloring $f$.
\end{proof}

\section{Concluding remarks}

In this paper we defined a new colored connection parameter
srd-number for connected graphs. We hope that with this new
parameter, avoiding the drawback of the parameter rd-number,
one could get a colored version of the famous Menger's Min-Max Theorem.
We do not know if this srd-number is actually equal to the
rd-number for every connected graph, and then posed a conjecture to
further study on the two parameters.
The results in the last sections fully
support the conjecture.

\end{document}